\documentclass[12pt]{article}
\usepackage{latexsym, amssymb}
\usepackage{dsfont}
\usepackage{graphicx} 
\usepackage{epsfig}
\usepackage{amsthm}

\makeatletter
\def\eqnarray{\stepcounter{equation}\let\@currentlabel=\theequation
\global\@eqnswtrue
\tabskip\@centering\let\\=\@eqncr
$$\halign to \displaywidth\bgroup\hfil\global\@eqcnt\z@
  $\displaystyle\tabskip\z@{##}$&\global\@eqcnt\@ne
  \hfil$\displaystyle{{}##{}}$\hfil
  &\global\@eqcnt\tw@ $\displaystyle{##}$\hfil
  \tabskip\@centering&\llap{##}\tabskip\z@\cr}

\def\endeqnarray{\@@eqncr\egroup
      \global\advance\c@equation\m@ne$$\global\@ignoretrue}

\def\@yeqncr{\@ifnextchar [{\@xeqncr}{\@xeqncr[5pt]}}
\makeatother

\newtheorem{lemma}{Lemma}[section]
\newtheorem{thm}[lemma]{Theorem}
\newtheorem{cor}[lemma]{Corollary}

\newtheorem{prop}[lemma]{Proposition}

\begin{document}

\newcounter{tellerr}
\renewcommand{\thetellerr}{(\roman{tellerr})}
\newenvironment{tabeleq}{\begin{list}%
{\rm  (\roman{tellerr})\hfill}{\usecounter{tellerr} \leftmargin=1.1cm
\labelwidth=1.1cm \labelsep=0cm \parsep=0cm}
                         }{\end{list}}

\newcommand{\Ni}{\mathds{N}}
\newcommand{\Qi}{\mathds{Q}}
\newcommand{\Ri}{\mathds{R}}
\newcommand{\Ci}{\mathds{C}}
\newcommand{\Zi}{\mathds{Z}}

\newcommand{\RRe}{\mathop{\rm Re}}
\newcommand{\IIm}{\mathop{\rm Im}}
\newcommand{\Tr}{{\mathop{\rm Tr \,}}}

\newcommand{\one}{\mathds{1}}

\hyphenation{groups}
\hyphenation{unitary}

\newcommand{\cl}{{\cal L}}

\thispagestyle{empty}

\vspace*{1cm}
\begin{center}
{\Large\bf Analytical aspects of isospectral drums} \\[4mm]

\large W. Arendt$^1$, A.F.M. ter Elst$^2$ and J.B. Kennedy$^1$

\end{center}

\vspace{4mm}

\begin{center}
{\bf Abstract}
\end{center}

\begin{list}{}{\leftmargin=1.8cm \rightmargin=1.8cm \listparindent=10mm 
   \parsep=0pt}
\item
We reexamine the proofs of isospectrality of the counterexample domains to Kac' 
question `Can one hear the shape of a drum?' from an analytical viewpoint.
We reformulate isospectrality in a more abstract setting as the existence of a 
similarity transform intertwining two operators associated with elliptic forms, 
and give several equivalent characterizations of this property as
intertwining the forms and form domains, the associated operators and operator 
domains, and the semigroups they generate.
On a representative pair of counterexample domains, we use these criteria to 
show that the similarity transform intertwines not only the Laplacians with Neumann 
(or Dirichlet) boundary conditions but also any two appropriately defined elliptic 
operators on these domains, even if they are not self-adjoint. 
However, no such transform can intertwine these operators 
if Robin boundary conditions are imposed instead of Neumann or Dirichlet.
We also remark on various operator-theoretic properties of such intertwining 
similarity transforms.
\end{list}

\vspace{6mm}


\vspace{6mm}

\noindent
{\bf Home institutions:}    \\[3mm]
\begin{tabular}{@{}cl@{\hspace{10mm}}cl}
1. & Institute of Applied Analysis & 
2. & Department of Mathematics  \\
& University of Ulm & 
  & University of Auckland  \\
& Helmholtzstr.\ 18  &
  & Private bag 92019 \\
& D-89069 Ulm  &
   & Auckland 1142 \\ 
& Germany  &
  & New Zealand \\[8mm]
\end{tabular}

\newpage
\setcounter{page}{1}

\section{Introduction} \label{Sdrumce1}

It took 30 years for Kac' famous question `Can one hear the shape of a drum?' \cite{Kac} 
to find an answer.
Gordon, Webb and Wolpert \cite{GWW} constructed two non-congruent planar domains 
whose Laplacians with Dirichlet (or Neumann) boundary conditions are isospectral, that is, 
they have the same sequence of eigenvalues, counted with multiplicities.
The standard counterexample takes the form of two polygons obtained by 
stitching together seven copies of a given non-equilateral triangle in two different ways.
These domains are manifestations in the plane of a general principle first enunciated by 
Sunada \cite{Sun1} and developed by B\'erard \cite{Ber1}, which to the best of our 
knowledge accounts for all known isospectral pairs, and which was used in \cite{GWW}.
Namely, if $H$ and $K$ are two subgroups of a finite group $G$, then a unitary 
intertwining operator between the spaces $L_2(H\setminus G)$ and $L_2(K\setminus G)$ 
induces an isometry between appropriate subspaces of any Hilbert space on which $G$ 
acts unitarily.
Subsequent to the publication of \cite{GWW}, several mathematicians, for example 
B\'erard \cite{Ber2}, Buser--Conway--Doyle--Semmler \cite{BCDS} and Chapman 
\cite{Cha}, gave simplified and more accessible proofs of the isospectrality of such 
domains.
The argument in all these expository proofs consists in 
showing that an eigenfunction on the first polygon
can be transposed to an eigenfunction on the second by taking particular linear 
combinations of its values on the seven equal constituent triangles, and vice versa.

Following the approach taken by B\'erard \cite{Ber2}, 
if we consider $L_2(\Omega_1)$ and $L_2(\Omega_2)$ rather as $L_2(T)^7$, 
where $T$ is the basic triangle (`brique fondamentale') and $\Omega_1$ and 
$\Omega_2$ the polygons (see Figure~\ref{fdrumce1}), 
\begin{figure}[ht!] \label{fdrumce1}         
\vspace*{5mm} 
\centering
\epsfig{file=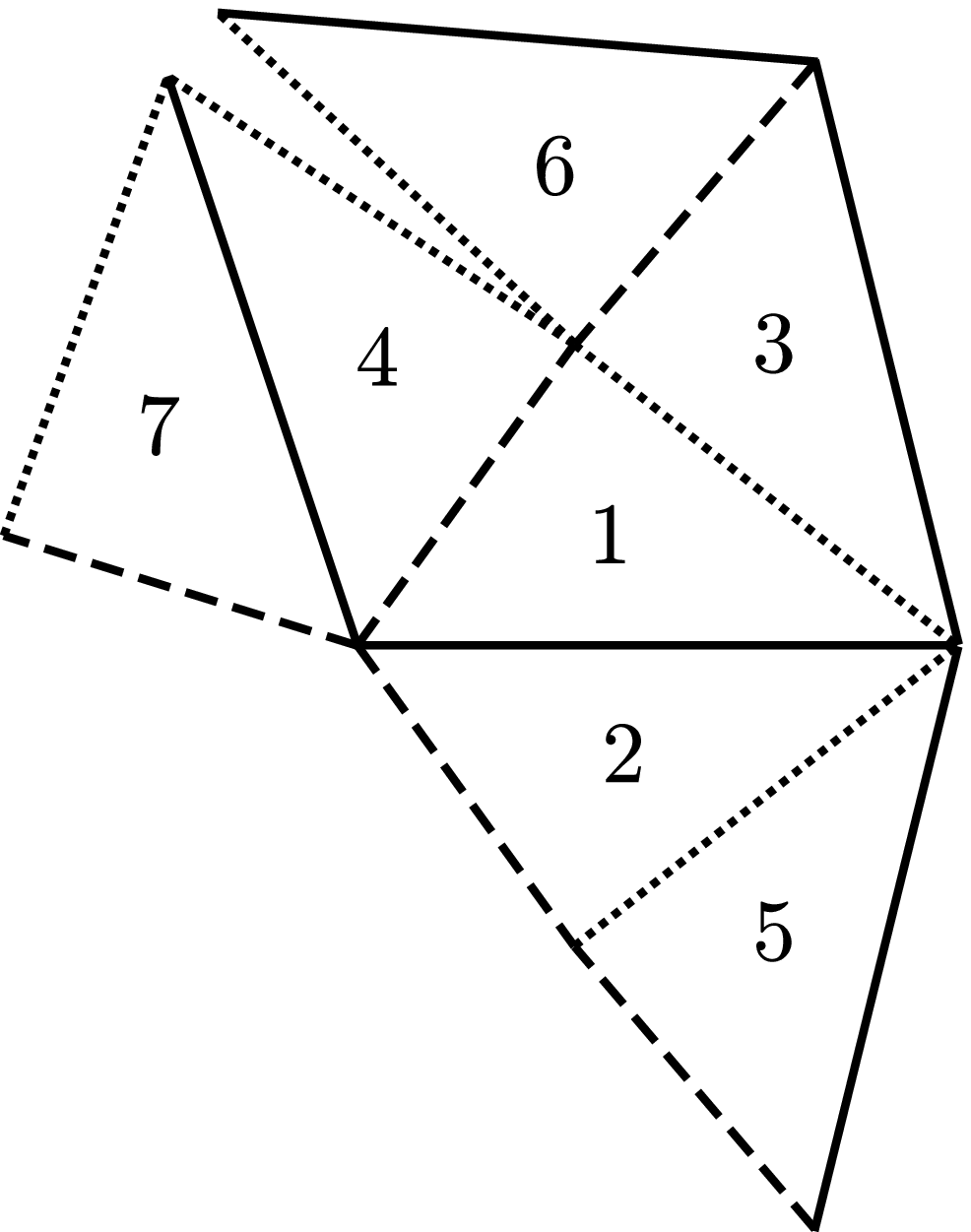, height=50mm} \hspace{25mm}
\epsfig{file=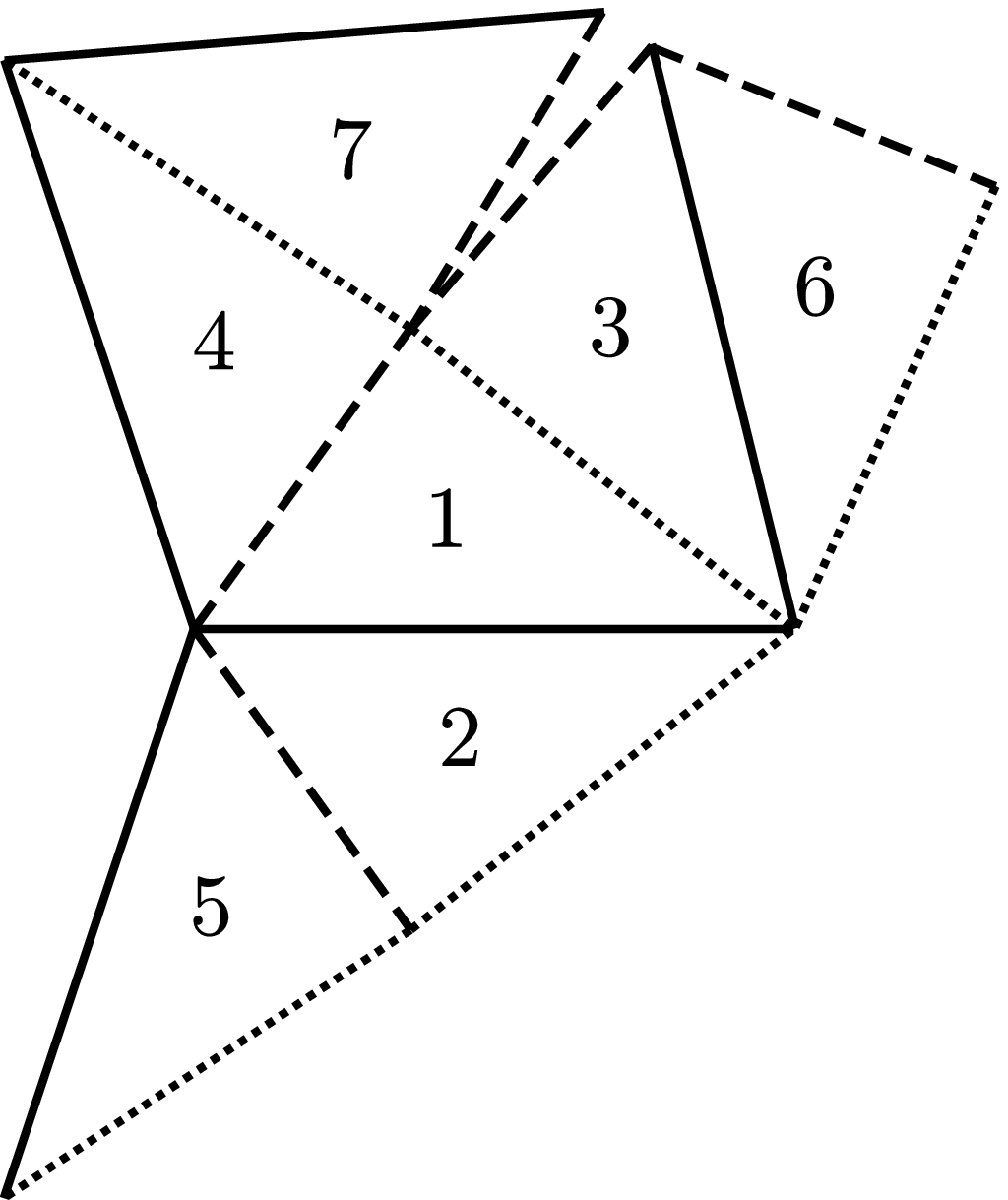, height=50mm}
\mbox{}
\makebox[0pt]{\raisebox{10mm}[0pt][0pt]{\mbox{}\hspace*{-210mm}$\Omega_1$ }}
\makebox[0pt]{\raisebox{10mm}[0pt][0pt]{\mbox{}\hspace*{-20mm}$\Omega_2$ }}
\caption{Two isospectral domains composed of seven isometric triangles. These are based 
on the `warped propeller' domains of \cite{BCDS}.}
\end{figure}
then we can construct an isometry $\Phi$ on $L_2(T)^7$ induced by a $7 \times 7$ matrix 
$B$ of scalars acting as a family of Euclidean isometries superimposing the seven triangles.
The core of the argument is that $\Phi$ restricts to an isometry mapping the Sobolev 
space $H^1_0(\Omega_1)$ onto $H^1_0(\Omega_2)$, whilst its adjoint $\Phi^*$
maps $H^1_0(\Omega_2)$ onto $H^1_0(\Omega_1)$.
Isospectrality of the Laplacians then follows from the variational characterization of 
the eigenvalues.
The later works of Buser {\it{et al}} \cite{BCDS} and Chapman \cite{Cha} motivate and 
describe in lay terms how $\Phi$ acts purely as a map between eigenfunctions, without 
touching upon the concept of isometric Sobolev spaces.

The aim of this paper is to reconsider these arguments from a more 
analytical perspective.
Rather than transposing eigenfunctions, we construct $\Phi$ as a similarity 
transform intertwining the realizations of the Laplacian with 
Neumann (or Dirichlet) boundary conditions, 
or equivalently, the semigroups generated by these realizations, 
on the respective polygons.
Moreover, we consider the operator-theoretic properties of such a transform $\Phi$ 
more carefully.
We will give a general characterization of maps $\Phi$ that intertwine any two 
operators associated with elliptic forms.
In light of this characterization it looks like a miracle that there exists 
a matrix which fulfils the criterion, but the key point is rather that $\Phi$ 
and its adjoint respect the form domains $H^1(\Omega_1)$ and $H^1(\Omega_2)$.
That the transform $\Phi$ intertwines the elliptic forms implies 
that it also intertwines the associated operators and semigroups.
In the case of the Laplacian, this is then equivalent to the isospectral property.
But it is not necessary that we consider only the Laplacian: since the Sobolev spaces 
are intertwined, any elliptic operator on $L_2(T)$, even if it is not self-adjoint, will yield 
two operators on $L_2(\Omega_1)$ and $L_2(\Omega_2)$ which are similar.
In place of isospectrality, the correct setting is now that of similarity, a stronger 
property in the non-self-adjoint case.

We can consider the Laplacian with Robin boundary conditions in this setting.
The question as to whether there exist isospectral pairs for the third boundary condition 
just as for the first and second seems to be a natural one, and was briefly mentioned in 
the survey article \cite{Pro}; but otherwise it 
appears to have received little attention, and no answer.
We will show that any operator acting as a family of superimposing 
isometries that intertwines the Robin Laplacians on $\Omega_1$ and $\Omega_2$ 
must also simultaneously intertwine the Dirichlet {\em and} Neumann Laplacians, 
which is easily shown to be impossible.
Thus there is no reason to suppose that any known pairs of domains which are 
Dirichlet or Neumann isospectral are also Robin isospectral, and it is an open 
question as to whether there exists {\em any} noncongruent pair of Robin isospectral domains.
The striking implication is that isospectrality could be essentially related to the 
boundary conditions and not the coefficients of the operators being intertwined.
Thus it may well be the case that one can hear the shape of a 
drum after all, if one loosens the membrane before striking it.

There is another motivation for studying similarity transforms within our framework.
It has been shown (see \cite{Are3}, and cf.~also \cite{ABE} and \cite{AE4} 
for the case of Riemannian manifolds) that two (Lipschitz) domains are 
necessarily congruent if there exists an order isomorphism 
intertwining the Laplacians.
Thus, in our case, the similarity transform $\Phi$ is not an order 
isomorphism, even though, at least in the case of Neumann boundary 
conditions, $\Phi$ may be taken as a positive linear map.
What goes wrong is that $\Phi$ is no longer disjointness preserving, 
as on each triangle it adds (the function values on) several distinct 
triangles together; thus $\Phi$ may be written as a finite sum of order 
isomorphisms, and due to this `mixing' property, $\Phi^{-1}$ is not positive.
Understanding and seeking to narrow the operator-theoretic gap between 
the characterization of such positive results as in \cite{Are3} and the negative 
counterexamples may help us to understand Kac' problem better, as well as offering 
an alternative approach to the standard one via heat and wave traces.

In fact, a version of Kac' question is still open.
The results here, just as those of \cite{GWW} and the other expositions, 
can be interpreted as saying that these seven triangles can be put together 
in two different ways to induce isomorphic Sobolev spaces, which is of course 
the essential idea behind B\'erard's version of Sunada's Theorem. 
With trivial and obvious modifications, the same is true for all 
other known counterexamples as presented in \cite{BCDS} (which are all based 
on Sunada's Theorem in the same way, and which can all be analyzed within our 
framework in the obvious way).
But the point is that the phenomenon exhibited by these domains is 
somehow exceptional and does not really answer Kac' question.
If we interpret the `correct' setting for Kac' question as being 
$C^\infty$-domains in the plane, then the question is 
still wide open; there is no known counterexample among $C^1$ or 
convex planar domains.
In four dimensions there is a counterexample of two non-congruent 
convex domains, given by Urakawa \cite{Ura} in 1982, which was in fact 
the first Euclidean example.
However, the issue of regularity of the boundary seems to be far more 
than a technicality, a point also made in the survey article of Protter
\cite{Pro}.
While it is certainly clear that any counterexamples generated via the 
principle of Sunada's Theorem must have corners, there are 
also remarkable and profound positive results obtained by 
Zelditch \cite{Zel} \cite{Zel2}, who proves that, within a certain class of 
domains in $\Ri^2$ with analytic boundary and certain symmetry conditions, 
any two isospectral domains are congruent. 
The presence of corners in a domain also has significant consequences for the 
asymptotic behaviour of the eigenvalues; for example, the curvature of 
the boundary appears in all terms of the asymptotic expansion of the 
heat kernel about $t=0$ (see, e.g., \cite{Wat}).

This article is organized as follows.
We start in Section~\ref{Sdrumce2} by characterizing operators which 
intertwine two semigroups generated by sectorial forms, and relating 
this to the isospectral property.
We phrase many of our results in the language of semigroups, as this 
allows us to work on $L_2$-spaces in place of the more abstruse operator domains. 
In Section~\ref{Sdrumce3} we recast the arguments given in \cite{Ber2} and \cite{BCDS} 
within this framework, showing first how we can decompose the large domains 
$\Omega_1$ and $\Omega_2$ into their constituent triangles, and give conditions 
allowing us to merge the associated Sobolev spaces together.
In this setting we then 
prove that realizations of the Neumann Laplacian on the two non-congruent polygons 
in Figure~\ref{fdrumce1} are similar. 
We work principally with the Neumann case as there are fewer conditions on the 
Sobolev spaces involved, and as the similarity transform and associated matrix have the 
particularly nice property that they may be taken to be positive.
In Section~\ref{Sdrumce4} we discuss properties of the intertwining operator $\Phi$ 
constructed in Section~\ref{Sdrumce3} from a more analytical, operator-theoretic 
perspective. 
The results in Section~\ref{Sdrumce3} are extended to more general elliptic 
operators in Section~\ref{Sdrumce5}.
We then consider Dirichlet boundary conditions in Section~\ref{Sdrumce6}.
The underlying ideas are the same, but the details of the construction 
turn out to be a little more complicated than in the Neumann case.
We therefore limit ourselves to indicating the differences 
vis-\`a-vis the Neumann Laplacian.
Finally, in Section~\ref{Sdrumce7}, we show that these arguments cannot be 
extended to Robin boundary conditions.

\section{Forms and intertwining operators} \label{Sdrumce2}

We start by introducing some basic terms and results from the theory of 
sectorial forms.
The idea is to consider equivalent formulations of isospectrality for the 
Dirichlet and Neumann Laplacians which are more suitable for adaptation 
to more general operators.
To that end, let $H$ and $V$ be complex Hilbert spaces such that 
$V$ is densely embedded in $H$.
Let $a \colon V \times V \to \Ci$ be a continuous sesquilinear form.
Assume that $a$ is {\bf elliptic}, that is, there exist $\omega \in \Ri$
and $\mu > 0$ such that 
\begin{equation}
\RRe a(u,u) + \omega \, \|u\|_H^2 
\geq \mu \, \|u\|_V^2
\label{eSdrumce2;1}
\end{equation}
for all $u \in V$.
Denote by $A$ the operator associated with $a$.
That is, the domain of $A$ is given by
\[
D(A) = \{ u \in V : \mbox{there exists an $f \in H$ such that }
    a(u,v) = (f,v)_H \mbox{ for all } v \in V \}, 
\]
and $A u = f$ for all $u \in D(A)$ and $f \in H$ such that 
$a(u,v) = (f,v)_H$ for all $v \in V$.
Then $-A$ generates a holomorphic semigroup on $H$.

We are of course particularly interested in the Dirichlet and Neumann Laplacians 
on $H = L_2(\Omega)$, where $\Omega \subset \Ri^d$ is an open set with 
finite measure.
These are self-adjoint operators with compact resolvent, and can 
be characterized as follows.
We omit the standard proof.

\begin{prop} \label{pdrumce200}
Let $A$ be an operator in a separable infinite dimensional Hil\-bert space~$H$.
The following are equivalent.
\begin{tabeleq}
\item \label{pdrumce200-1}
$A$ is self-adjoint, bounded from below and has compact resolvent.
\item \label{pdrumce200-2}
There exist a Hilbert space $V$ which is densely and compactly embedded in 
$H$ and a symmetric, continuous elliptic form $a \colon  V\times V \to \Ci$ such that 
$A$ is associated with $a$.
\item \label{pdrumce200-3}
There exist an orthonormal basis $(e_n)_{n \in \Ni}$ of $H$ and an
increasing sequence of real numbers $(\lambda_n)_{n \in \Ni}$ with $\lim_{n\to\infty} 
\lambda_n =\infty$ such that
\[
D(A)=\{u\in H: \sum_{n=1}^\infty |\lambda_n \, (u,e_n)_H|^2 < \infty\}
\]
and $Au = \sum_{n=1}^\infty \lambda_n \, (u,e_n)_H \, e_n$ for all $u \in D(A)$.
\end{tabeleq}
\end{prop}

If $A$ satisfies these equivalent conditions, we call $\lambda_n$ the {\bf $n$-th 
eigenvalue of $A$} and $(\lambda_n)_{n \in \Ni}$ the sequence of eigenvalues of $A$, 
where repetition is possible.

\medskip

Let us now assume that we have two forms $a_1$ and $a_2$, with dense
form domains $V_1$ and $V_2$ in Hilbert spaces $H_1$ and $H_2$, respectively. 
We assume throughout that both $a_1$ and $a_2$ are continuous and elliptic.
Let $A_1$ and $A_2$ be the operators associated with $a_1$ and $a_2$, which are 
automatically bounded from below thanks to the ellipticity assumption.
Denote by $S^1$ and $S^2$ the semigroups generated by 
$-A_1$ and $-A_2$.
If $A_1$ and $A_2$ also are self-adjoint and have compact resolvent then we call them {\bf isospectral} 
if they have the same sequence of eigenvalues.
In this case we will denote by $(e_n)_{n \in \Ni}$ the sequence of (normalized) eigenfunctions 
of $A_1$ on $H_1$ and by $(f_n)_{n \in \Ni}$ the similarly normalized eigenfunctions of 
$A_2$ on $H_2$.
It is then immediate that there exists a unitary operator $U \in \cl(H_1,H_2)$ such that
\begin{equation} \label{pdrumce201b}
U^{-1} \, S^2_t \, U = S^1_t,
\end{equation}
for all $t>0$.
We may simply choose $U$ such that $U e_n = f_n$ for all $n \in \Ni$.
If we assume $H_i = L_2(\Omega_i)$ for some open $\Omega_i \subset \Ri^d$
and $A_i$ is the Dirichlet (or Neumann) Laplacian on $\Omega_i$ for all $i \in \{ 1,2 \} $,
then Kac' question may be phrased as asking whether the existence of an intertwining 
operator as in (\ref{pdrumce201b}) implies the existence of an isometry 
$\tau \colon  \Omega_1 \to \Omega_2$. 

However, we wish to consider more general operator-theoretic notions 
than isospectrality, in particular allowing for non-self-adjoint operators.
Moreover, the similarity transform $\Phi$ that we construct in Section~\ref{Sdrumce3} 
is, in general, not unitary.
See also Section~\ref{Sdrumce4}. 
(This assertion is also true for the equivalent constructions in \cite{Ber2} and 
\cite{BCDS}, where the mechanism is of course the same.)
The next proposition gives a general characterization of an operator $\Phi \colon  H_1 \to H_2$ 
that intertwines $A_1$ and $A_2$ in terms of the forms $a_1$ and $a_2$.
Of particular interest to us in what follows is Condition~\ref{pdrumce202-3}.
We do not require $A_1$ or $A_2$ to be self-adjoint or have compact resolvent.

\begin{prop} \label{pdrumce202}
Let $\Phi \in \cl(H_1,H_2)$.
Consider the following conditions.
\begin{tabeleq}
\item \label{pdrumce202-1}
$S^2_t \, \Phi = \Phi \, S^1_t$ for all $t > 0$.
\item \label{pdrumce202-2}
$\Phi(D(A_1)) \subset D(A_2)$ and $A_2 \, \Phi u = \Phi \, A_1 u$
for all $u \in D(A_1)$.
\item \label{pdrumce202-3}
$\Phi(V_1) \subset V_2$, $\Phi^*(V_2) \subset V_1$ and 
$a_2(\Phi u, v) = a_1(u, \Phi^* v)$ for all $u \in V_1$ and $v \in V_2$.
\end{tabeleq}
Then {\rm \ref{pdrumce202-1}}$\Leftrightarrow${\rm \ref{pdrumce202-2}}$\Leftarrow${\rm \ref{pdrumce202-3}}.
\end{prop}
\begin{proof}
`\ref{pdrumce202-1}$\Rightarrow$\ref{pdrumce202-2}'.
Let $u \in D(A_1)$.
Then 
\[
\Phi \, A_1 u
= \lim_{t \downarrow 0} t^{-1}  \Phi \, (I - S^1_t) u
= \lim_{t \downarrow 0} t^{-1}  (I - S^2_t) \, \Phi u
= A_2 \, \Phi u
.  \]
So $\Phi u \in D(A_2)$ and $A_2 \, \Phi u = \Phi \, A_1 u$.

`\ref{pdrumce202-2}$\Rightarrow$\ref{pdrumce202-1}'.
Replacing $A_k$ by $\omega I + A_k$, we may assume that both 
$S^1$ and $S^2$ are exponentially decreasing.
Let $u \in D(A_1)$ and $\lambda > 0$.
Then $(\lambda I + A_2) \, \Phi u = \Phi \, (\lambda I + A_1) u$.
Since $\lambda I + A_1$ is surjective, it follows that 
$\Phi \, (\lambda I + A_1)^{-1} v = (\lambda I + A_2)^{-1} \, \Phi v$
for all $v \in H_1$.
Hence by iteration,
$\Phi \, (\lambda I + A_1)^{-n} = (\lambda I + A_2)^{-n} \, \Phi$
for all $n \in \Ni$.
Then by (7) in \cite{Yos} Section~IX.7 one deduces \ref{pdrumce202-1}.

`\ref{pdrumce202-3}$\Rightarrow$\ref{pdrumce202-2}'.
Let $u \in D(A_1)$.
Then for all $v \in V_2$ one has $\Phi v \in V_1$ and 
\[
a_2(\Phi u, v) 
= a_1(u, \Phi^* v)
= (A_1 u, \Phi^* v)_{H_1}
= (\Phi \, A_1 u, v)_{H_2}
.  \]
Hence $\Phi u \in D(A_2)$ and $A_2 \, \Phi u = \Phi \, A_1 u$.
\end{proof}

We remark that under additional assumptions on the operators $A_1$ and $A_2$ it can be 
proved that all three statements in Proposition~\ref{pdrumce202} are equivalent.
It suffices that $a_1$ and $a_2$ have the {\bf square root property} on $H$,  which 
means that for the square root operator $(\omega_k I + A_k)^{1/2}$ of $\omega_k I + A_k$,
defined as in \cite{ABHN}, Section~3.8, where $\omega_k$ is the constant in 
(\ref{eSdrumce2;1}), we have $D((\omega_k I + A_k)^{1/2}) = V_k$, for all $k \in \{ 1,2 \} $. 
This is not always the case; a counterexample has been given by McIntosh \cite{McI1}, 
although it is always true if $A_1$ and $A_2$ are self-adjoint. 
As we do not need this equivalence in the sequel, we do not go into details.

\medskip

We  next assume that the intertwining operator $\Phi \in \cl (H_1,H_2)$ is 
invertible and thus an isomorphism between $H_1$ and $H_2$. 

\begin{cor} \label{cdrumce203}
Let $\Phi \in \cl(H_1,H_2)$ be invertible. 
Consider the following statements.
\begin{tabeleq}
\item \label{cdrumce203-1}
$\Phi(D(A_1)) \subset D(A_2)$ and $A_2 \, \Phi u = \Phi \, A_1 u$
for all $u \in D(A_1)$.
\item \label{cdrumce203-2}
$\Phi(D(A_1)) = D(A_2)$ and $A_2 \, \Phi u = \Phi \, A_1 u$ for all $u \in D(A_1)$.
\item \label{cdrumce203-3}
$\Phi^{-1} \, S^2_t \, \Phi = S^1_t$ for all $t > 0$.
\item \label{cdrumce203-4}
If $u\in D(A_1)$ and $\lambda \in \Ri$ are such that $A_1 u=\lambda u$, then 
$\Phi u \in D(A_2)$ and $A_2\Phi u = \lambda \Phi u$.
\end{tabeleq}
Then {\rm \ref{cdrumce203-1}} $\Leftrightarrow$ {\rm \ref{cdrumce203-2}} 
$\Leftrightarrow$ {\rm \ref{cdrumce203-3}} $\Rightarrow$ {\rm \ref{cdrumce203-4}}. 
If in addition $A_1$ is self-adjoint and has compact resolvent, then all four statements 
are equivalent.
\end{cor}

We say that $A_1$ and $A_2$ are {\bf similar}, or equivalently, that the semigroups
$S^1$ and $S^2$ are {\bf similar}, if the equivalent statements 
\ref{cdrumce203-1}--\ref{cdrumce203-3} hold.
In the case where $A_1$ and $A_2$ are self-adjoint and have compact resolvent, 
we may replace \ref{cdrumce203-2} with the statement `$\Phi(D(A_1)) = 
D(A_2)$ and the spectra of $A_1$ and $A_2$ coincide'. 
Thus we may regard 
similarity as a more general property than isospectrality.

The next result was stated in \cite{Are3} Lemma~1.3 for self-adjoint operators, 
but we note that it is also a direct consequence of Proposition~\ref{pdrumce202} and 
Corollary~\ref{cdrumce203} without requiring this assumption.

\begin{cor} \label{cdrumce204}
Let $\Phi \in \cl(H_1,H_2)$ be unitary.
Then the following are equivalent.
\begin{tabeleq}
\item \label{cdrumce204-1}
$S^2_t \, \Phi = \Phi \, S^1_t$ for all $t > 0$.
\item \label{cdrumce204-2}
$\Phi(V_1) = V_2$ and 
$a_2(\Phi u, \Phi v) = a_1(u, v)$ for all $u,v \in V_1$.
\end{tabeleq}
\end{cor}

We finish this section by pointing out that the existence of a unitary similarity 
transform is guaranteed by self-adjointness of the operators alone, and compactness 
of the resolvents is not needed.

\begin{prop} \label{pdrumce205}
Let $A_1$ and $A_2$ be two self-adjoint operators on $H_1$ and $H_2$, 
respectively.
Assume that the semigroups $S^1$ and $S^2$ are similar.
Then there exists a unitary operator $U \in \cl(H_1,H_2)$ such that
\[
U^{-1} \, S^1_t \, U = S^2_t
\]
for $t>0$.
\end{prop}
\begin{proof}
We consider the polar decomposition $\Phi = U \, |\Phi|$, where $U \in \cl(H_1,H_2)$ 
is unitary and $|\Phi| = (\Phi^* \Phi)^{1/2} \in \cl(H_1)$ is invertible and self-adjoint.
Since
\[
\Phi^* S^2_t = (S^2_t \, \Phi)^* = (\Phi \, S^1_t)^* = S^1 \, \Phi^*
\]
for all $t>0$, we see that $\Phi^*$ is also an intertwining operator.
Thus $|\Phi|$ commutes with $S^1_t$ for all $t>0$, and so
\[
U \, S^1_t 
= U \, |\Phi| \, |\Phi|^{-1} \, S^1_t 
= \Phi \, S^1_t \, |\Phi|^{-1} 
= S^2_t \, \Phi \, |\Phi|^{-1} 
= S^2_t \, U 
\]
for all $t > 0$.
\end{proof}

\section{Isospectral domains for the Neumann Laplacian} \label{Sdrumce3}

For an open polygon $\Omega$ in $\Ri^2$ we denote by $\Delta_\Omega^N$ the 
Neumann Laplacian on $L_2(\Omega)$. 
This realization of the Laplacian with Neumann boundary conditions is self-adjoint, 
has compact resolvent and its negative is bounded from below, with a sequence of 
eigenvalues $0 = \lambda_0 \leq \lambda_1 \leq \ldots \to \infty$.
We will consider the two (very) warped propeller-like domains from 
Figure~\ref{fdrumce1} and show that $\Delta_{\Omega_1}^N$ and 
$\Delta_{\Omega_2}^N$ are similar.
This will be done with the help of our form criterion established in 
Proposition~\ref{pdrumce202}.
As a corollary we deduce that $\Delta_{\Omega_1}^N$ and $\Delta_{\Omega_2}^N$ 
are isospectral even though $\Omega_1$ and $\Omega_2$ are obviously not congruent.
Note that $\Omega_1$ and $\Omega_2$ look like propellers if the 
constituent triangles are equilateral.

Since we wish to decompose our polygons into their constituent triangles, we need to start 
with some basic facts about traces and integration by parts.
We let $\Omega \subset \Ri^2$ be an arbitrary open polygon, although in practice we only 
need the following results for our warped propellers.
On the boundary $\Gamma$ of $\Omega$ we let $\sigma$ denote the usual surface measure; 
on each straight line segment, $\sigma$ is simply one-dimensional Lebesgue measure. 
The Trace Theorem states that there exists a unique bounded operator $\Tr \colon  H^1(\Omega) 
\to L_2(\Gamma)$ such that $\Tr(u)=u|_\Gamma$ for all $u \in H^1(\Omega) \cap 
C(\overline \Omega)$.
Observe that, since $\Omega$ is Lipschitz, the space $H^1(\Omega) \cap C(\overline \Omega)$ is 
dense in $H^1(\Omega)$.
By $\nu(z)=(\nu_1(z), \nu_2(z))$ we denote the outer unit normal to $\Omega$ at $z\in 
\Gamma$.
Then $\nu(z)$ is constant on each straight line segment of the boundary. 
The integration by parts formula states that
\[
-\int_\Omega (\partial_j u) \,v = \int_\Omega u\,\partial_j v - 
\int_\Gamma \nu_j \, u \,v
\]
for all $u,v \in H^1(\Omega)$ and $j \in \{ 1,2 \} $. 
Here the integral over $\Gamma$ is with respect to $\sigma$, and we have omitted the 
trace to simplify notation.

The {\bf Neumann Laplacian} is by definition the operator $\Delta_\Omega^N$ on $L_2(\Omega)$ 
such that $- \Delta_\Omega^N$ is associated with the form $a \colon H^1(\Omega) \times H^1(\Omega) \to \Ci$ given by
\begin{equation}
a(u,v) = \int_\Omega \nabla u \cdot \overline{\nabla v}.
\label{eSdrumse2;8}
\end{equation}
We denote by $S$ the semigroup generated by $\Delta_\Omega^N$.

If $u \in H^1(\Omega)$ is such that the distributional Laplacian
$\Delta u \in L_2(\Omega)$, then for all 
$h \in L_2(\Gamma)$ we say that $\partial_\nu u = h$ if
\begin{equation}
\int_\Omega (\Delta u) \,v + \int_\Omega \nabla u\cdot \nabla v= \int_\Gamma h \,v
\label{eSdrumce3;0}
\end{equation}
for all $v\in H^1(\Omega)$. That is, we define the normal derivative via Green's 
formula.
Based on this definition, the operator $\Delta_\Omega^N$ has the domain
\[
D(\Delta_\Omega^N)=\{u\in H^1(\Omega): \Delta u\in L_2(\Omega) \mbox{ and } \partial_\nu u=0\}.
\]
This is valid whenever $\Omega$ is a Lipschitz domain.

Now let $T$ be a fixed scalene triangle whose three different sides are labelled $\Gamma_1$, 
$\Gamma_2$ and $\Gamma_3$ as in Figure~\ref{fdrumce2}.
\begin{figure}[ht!]          
\centering
\vspace*{1mm} 
\epsfig{file=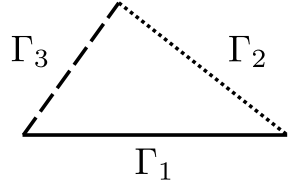, height=24mm} 
\caption{The triangle $T$.\label{fdrumce2}}
\end{figure}
Thus if $\Omega_1$ and $\Omega_2$ are broken into their seven constituent triangles, 
then each is congruent to $T$.

Now let $\Omega$ be either one of $\Omega_1$ and $\Omega_2$ and consider the seven
open disjoint triangles $T_1,\ldots, T_7$ such that
\[
\overline \Omega = \bigcup_{k=1}^7 \overline T_k.
\]
Two triangles $\overline T_k$ and $\overline T_l$ may have a common side; there are six such 
sides inside $\Omega$.

If $u\in H^1(\Omega)$, then $u_k:= u|_{T_k} \in H^1(T_k)$ for all $k \in \{ 1,\ldots,7 \} $.
Conversely, the following basic result holds.

\begin{lemma} \label{ldrumce3b03}
Let $u \in L_2(\Omega)$ be such that $u_k:=u|_{T_k} \in H^1(T_k)$ for all $k \in \{ 1,\ldots,7 \} $.
Then $u \in H^1(\Omega)$ if and only if $u_k$ and $u_l$ have the same trace on common 
sides of $T_k$ and $T_l$ for all $k,l \in \{ 1,\ldots,7 \} $ with $k \neq l$.
Moreover, if $u \in H^1(\Omega)$ then
$(\partial_j u)|_{T_k} = \partial_j u_k$ on $T_k$ for all
$k \in \{ 1,\ldots,7 \} $ and $j \in \{ 1,2 \} $.
\end{lemma}
\begin{proof}
Since $H^1(\Omega) \cap C(\overline\Omega)$ is dense in $H^1(\Omega)$ the condition 
on the traces is clearly necessary.
Assume now that $u$ satisfies this trace condition. 
Let $\varphi \in C_c^1(\Omega)$ and $j \in \{1,2\}$. Then 
\begin{eqnarray}
-\int_\Omega u \,\partial_j \varphi
& = & -\sum_{k=1}^7 \int_{T_k} u_k \, \partial_j\varphi \nonumber \\
& = & \sum_{k=1}^7 \Big( \int_{T_k} (\partial_j u_k) \, \varphi 
        - \int_{\partial T_k} \nu_{k,j} \, u_k \, \varphi \Big) \nonumber \\
& = & \sum_{k=1}^7 \int_{T_k} (\partial_j u_k) \, \varphi \label{eldrumce3b03;1} \\
& = & \int_\Omega w \, \varphi, \nonumber
\end{eqnarray}
where $w \in L_2(\Omega)$ is such that $w|_{T_k} = \partial_j u_k$.
Thus $\partial_j u = w$ in $\Omega$ by definition of the weak derivative of a function.
Here we denote by $\nu_k$ the outer unit normal to $T_k$ on $\partial T_k$ with its two 
components $\nu_{k,1}$ and $\nu_{k,2}$.
Since $\nu_k = - \nu_l$ on $\overline T_k \cap \overline T_l$ whenever $k\neq l$, we have
\[
\sum_{k=1}^7 \int_{\partial T_k} \nu_{k,j} \, u \, \varphi = 0
,  \]
which we used in (\ref{eldrumce3b03;1}).
\end{proof}

If $\tau \colon  \Ri^2 \to \Ri^2$ is an isometry,
then the map $U \colon  L_2(\tau(\Omega)) \to L_2(\Omega)$ given by $u \mapsto u 
\circ \tau$ is a unitary operator, $U(H^1(\tau(\Omega))) = H^1(\Omega)$ and 
$U|_{H^1(\tau(\Omega))}$ is also unitary and isometric.

Now consider the first warped propeller $\Omega_1$ from Figure~\ref{fdrumce1}. 
Denote by $T_1,\ldots, T_7$ the seven disjoint triangles isomorphic to $T$ such that 
$\overline \Omega_1 = \bigcup_{k=1}^7 \overline T_k$, and 
for all $k \in \{ 1,\ldots,7 \} $ denote by $\tau_k$ the isometry mapping 
$T$ onto $T_k$.
If we define a map 
\[
\Phi_1(w) = (w|_{T_1} \circ \tau_1,\ldots, w|_{T_7} \circ \tau_7),
\]
then $\Phi_1 \colon  L_2(\Omega_1) \to L_2(T)^7$ is unitary and $\Phi_1(H^1(\Omega_1)) = V_1$, 
where
\begin{eqnarray*}
V_1 = \{ (u_1,\ldots,u_7) \in H^1(T)^7 
& : & u_1 = u_2 \mbox{ and } u_4 = u_7 \mbox{ on } \Gamma_1  \\
& & u_1 = u_3 \mbox{ and } u_2 = u_5 \mbox{ on } \Gamma_2  \\
& & u_1 = u_4 \mbox{ and } u_3 = u_6 \mbox{ on } \Gamma_3 \}
.
\end{eqnarray*}
Here we mean more precisely that $u_1$ and $u_2$ have the same trace on $\Gamma_1$, 
and so on. 
Since the trace is a continuous mapping from $H^1(T)$ into $L_2(\partial T)$, the space $V_1$ 
is closed in $H^1(T)^7$.

We now define a form $\tilde a_1 \colon  V_1 \times V_1 \to \Ci$ by
\begin{equation} \label{epdrumce3b01}
\tilde a_1 (u,v) := \sum_{k=1}^7 \int_T \nabla u_k \cdot \overline{\nabla v_k},
\end{equation}
where we have written $u = (u_1,\ldots,u_7)$ and $v = (v_1,\ldots,v_7)$.
Then $\tilde a_1$ is continuous, symmetric and elliptic with respect to $L_2(T)^7$.
We denote by $\widetilde A_1$ the self-adjoint operator on $L_2(T)^7$ associated with 
$\tilde a_1$ and by $\widetilde S^1$ the semigroup generated by $-\widetilde A_1$.
We next show that the operators $-\widetilde A_1$ and $\Delta_{\Omega_1}^N$ 
(and the semigroups they generate) are similar. 
Let $S^1$ be the semigroup generated by the Neumann Laplacian on $\Omega_1$. 
We denote by $a_1$ the form associated with the semigroup $S^1$ 
on $\Omega_1$, cf.\ (\ref{eSdrumse2;8}).

\begin{prop} \label{pdrumce3b03}
If $t>0$ then
\[
\Phi^{-1}_1 \, \widetilde S^1_t \, \Phi_1 = S^1_t
.  \]
\end{prop}
\begin{proof}
Let $u,v \in H^1(\Omega_1)$ and write $\Phi_1(u) = (u_1,\ldots, u_7)$ and
$\Phi_1(v) = (v_1,\ldots, v_7)$. 
Then
\begin{eqnarray}
\tilde a_1 (\Phi_1 u, \Phi_1 v)
& = & \sum_{k=1}^7 \int_T \nabla u_k \cdot \overline{\nabla v_k} \nonumber \\
& = & \sum_{k=1}^7 \int_T \nabla (u\circ \tau_k) \cdot \overline{\nabla(v\circ \tau_k)} \nonumber \\
& = & \sum_{k=1}^7 \int_T (\nabla u) \circ \tau_k \cdot \overline{(\nabla v) \circ \tau_k} \nonumber \\
& = & \sum_{k=1}^7 \int_{T_k} \nabla u \cdot \overline{\nabla v} \nonumber \\
& = & \int_{\Omega_1} \nabla u \cdot \overline{\nabla v}. \label{edrumce3b;1}
\end{eqnarray}
Now the claim follows from Corollary~\ref{cdrumce204}.
\end{proof}

An obvious analogue holds for $\Omega_2$. 
Namely, define a unitary map $\Phi_2 \colon  L_2(\Omega_2) \to L_2(T)^7$ in the obvious 
way, and let 
\begin{eqnarray*}
V_2 = \{ (u_1,\ldots,u_7) \in H^1(T)^7 
& : & u_1 = u_2 \mbox{ and } u_3 = u_6 \mbox{ on } \Gamma_1  \\
& & u_1 = u_3 \mbox{ and } u_4 = u_7 \mbox{ on } \Gamma_2  \\
& & u_1 = u_4 \mbox{ and } u_2 = u_5 \mbox{ on } \Gamma_3 \},
\end{eqnarray*}
so that $\Phi_2(H^1(\Omega_2)) = V_2$. 
We define $\tilde a_2 \colon  V_2 \times V_2 \to \Ci$ by 
\[
\tilde a_2(\Phi_2 u,\Phi_2 v):= \sum_{k=1}^7 \int_T \nabla u_k \cdot \overline{\nabla v_k},
\]
where $\Phi_2 u = (u_1,\ldots,u_7)$ for all $u \in H^1(\Omega_2)$, etc., and denote by 
$\widetilde A_2$ the self-adjoint operator on $L_2(T)^7$ associated with $\tilde a_2$.
Let $\widetilde S^2$ be the semigroup generated by $- \widetilde A_2$ 
and by $S^2$ the semigroup generated by $\Delta_{\Omega_2}^N$.

\begin{prop} \label{pdrumce3b04}
The operators $\widetilde A_2$ on $L_2(T)^7$ and $-\Delta_{\Omega_2}^N$ on 
$L_2(\Omega_2)$ are similar.
Precisely,
\[
\Phi_2^{-1} \, \widetilde S^2_t \, \Phi_2 = S^2_t
\]
for all $t > 0$.
\end{prop}

The similarities established so far are quite simple; analogous results hold for any polygon 
decomposed into triangles. 
But the attraction of this approach is that, to show that $\Delta_{\Omega_1}^N$ and 
$\Delta_{\Omega_2}^N$ are similar, it suffices to prove the statement for $\widetilde A_1$ 
and $\widetilde A_2$, which are both defined as operators on $L_2(T)^7$. 
This is exactly what we shall now do, and it is here that the special combinatorial relations 
defining $V_1$ and $V_2$ are crucial. 

We define a map $B \colon \Ri^7 \to \Ri^7$ by 
\[
B = \left( \begin{array}{ccccccc}
 0 & 1 & 1 & 1 & 0 & 0 & 0  \\
 1 & 0 & 1 & 0 & 0 & 0 & 1  \\
 1 & 0 & 0 & 1 & 1 & 0 & 0  \\
 1 & 1 & 0 & 0 & 0 & 1 & 0  \\
 0 & 0 & 0 & 1 & 0 & 1 & 1  \\
 0 & 1 & 0 & 0 & 1 & 0 & 1  \\
 0 & 0 & 1 & 0 & 1 & 1 & 0  
           \end{array} \right)_{\textstyle .}
\]
This is an analogue for our domains of the matrix $T^N$ considered in \cite{Ber2}
(with $a=0$ and $b=1$; see the discussion in Section~\ref{Sdrumce4}). 
Moreover, define $\Phi \colon L_2(T)^7 \to L_2(T)^7$ by 
\[
(\Phi u)_k = \sum_{l=1}^7 b_{kl} \, u_l
.  \]
The adjoint $\Phi^*$ may also be defined directly with respect to $B^*$ by
\[
(\Phi^* u)_l = \sum_{k=1}^7 b_{kl} \, u_k 
.  \]
It is a simple (but central) calculation to show that 
\[
\Phi(V_1) \subset V_2 
\quad \mbox{and} \quad \Phi^*(V_2) \subset V_1
.  \]
Moreover, for all $u \in V_1$ and $v \in V_2$ one has
\begin{eqnarray}
\tilde a_2(\Phi u, v)
=  \sum_{k=1}^7 \int_T \nabla(\Phi u)_k \cdot \overline{\nabla v_k}
& = & \sum_{k=1}^7 \sum_{l=1}^7 b_{kl} \, \int_T \nabla u_l \cdot 
	\overline{\nabla v_k} \nonumber \\
& = & \sum_{l=1}^7 \int_T \nabla u_l \cdot \overline{\nabla (\Phi^*v)_l}
= \tilde a_1(u, \Phi^* v)
. \label{eSdrumce4;1}
\end{eqnarray}
Using Proposition~\ref{pdrumce202} it follows that $\Phi \, \widetilde S^1_t = 
\widetilde S^2_t \, \Phi$ for all $t > 0$.
It is easy to verify that the matrix $B$ is invertible, implying that $\Phi$ is 
invertible as an operator on $L_2(T)^7$, and therefore 
\[
\widetilde S^1_t = \Phi^{-1} \, \widetilde S^2_t \, \Phi
\]
for all $t > 0$.
So the semigroups are similar.
If we define
\[
U = \Phi_2^{-1} \, \Phi \, \Phi_1
,  \]
then $U \colon  L_2(\Omega_1) \to L_2(\Omega_2)$ is an isomorphism such that 
\[
S^1_t = U^{-1} \, S^2_t \, U
\]
for all $t > 0$.
We have proved the following result.

\begin{thm} \label{tdrumce401}
The semigroups $S^1$ and $S^2$ are similar.
In particular, $\Delta^N_{\Omega_1}$ and $\Delta^N_{\Omega_2}$
are isospectral, i.e.\ they have the same sequence of
eigenvalues, even though $\Omega_1$ and $\Omega_2$ are not congruent.
\end{thm}

\section{Order properties of the similarity transform} \label{Sdrumce4}

We keep the notation of the previous section and consider the intertwining 
isomorphism $U \colon L_2(\Omega_1) \to L_2(\Omega_2)$ more closely.
Recall that a linear map $R \colon L_2(\Omega_1) \to L_2(\Omega_2)$
is called {\bf positive} if $f \geq 0$ implies $R f \geq 0$ for all 
$f \in L_2(\Omega_1)$.
We then write $R \geq 0$.
One calls $R$ a {\bf lattice homomorphism} if 
$R(f \vee g) = Rf \vee Rg$ for all $f,g \in L_2(\Omega_1,\Ri)$.
The map $R$ is called {\bf disjointness preserving} if 
$f \cdot g = 0$ a.e.\ implies $(Rf) \cdot (Rg) = 0$ a.e.\ for all 
$f,g \in L_2(\Omega_1)$.
It is well known that $R$ is a lattice homomorphism if and only if 
$R$ is positive and disjointness preserving.
Finally, we call $R$ an {\bf order isomorphism} or a 
{\bf lattice isomorphism} if $R$ is bijective and both 
$R$ and $R^{-1}$ are positive.
This is equivalent to $R$ being a bijective lattice homomorphism.

We recall from \cite{Are3} Theorem~3.20 the following result.

\begin{thm} \label{tdrumce402}
Let $\Omega_1,\Omega_2 \subset \Ri^d$ be two Lipschitz domains.
If there exists an order isomorphism
$U \colon L_2(\Omega_1) \to L_2(\Omega_2)$ such that 
\[
U \, S^1_t
= S^2_t \, U
\]
for all $t > 0$, then $\Omega_1$ and $\Omega_2$ are congruent.
\end{thm}

Here, as before, $S^1$ and $S^2$ are the semigroups generated 
by $\Delta_{\Omega_1}^N$ and $\Delta_{\Omega_2}^N$, respectively.
Using Theorem~\ref{tdrumce402} it follows that the similarity transform $U$ 
in Theorem~\ref{tdrumce401} is not an order isomorphism.
In fact, $U = \Phi_2^{-1} \, \Phi \, \Phi_1$.
Recall that $\Phi \colon L_2(T)^7 \to L_2(T)^7$ is 
given by the matrix $B$, which is clearly positive.
Thus $\Phi$ is a positive map.
Since $\Phi_1$ and $\Phi_2$ are order isomorphisms, 
$U$ is also positive.
Hence $\Phi^{-1}$, and equivalently also $U^{-1}$, is not positive.

It is easy to see that a map from $L_2(T)^7$ into $L_2(T)^7$
given by a matrix as above is disjointness preserving if and only if 
each row in the matrix has at most one nonzero entry.
By way of contrast, our matrix $B$ has three nonzero entries in each row.
It follows that our $\Phi$ is the sum of three lattice homomorphisms.
This shows directly that $\Phi$ and $U$ are not disjointness preserving.

Finally, we mention that the intertwining isomorphism $U$ and the 
matrix $B$ that induces it are not unique. If we let $\one$ denote 
the $7 \times 7$ matrix whose $(k,l)$th-entry is $1$ for all $k,l\in\{1,\ldots,7\}$, 
and define $\widehat B:= \alpha (\one-B) + \gamma B$, then it may be verified that 
$\widehat B$ gives rise in the same way to another intertwining 
isomorphism $\widehat \Phi$, provided that the coefficients $\alpha,\gamma \in \Ri$ 
satisfy basic non-degeneracy conditions.
Our original matrix $B$ and the similarity transform $\Phi$ that it induces 
are easily seen to be normal but not unitary.
We know from Section~\ref{Sdrumce2} that in such a case one can always find 
a unitary transform related to $\Phi$, for example, via the polar decomposition 
$\Phi = U \, |\Phi|$.
However, if we choose the coefficients $\alpha$ and $\gamma$ appropriately, 
namely, as a pair of simultaneous solutions to $4\alpha^2 + 3\gamma^2 = 1$ and 
$2\alpha^2 + 4\alpha\gamma + \gamma^2 = 0$, then it is easy to check that 
$(\widehat B)^* \, \widehat B = I$, that is, the matrix $\widehat B$ is unitary. 
In this case, the similarity transform associated with $\widehat B$ is also 
unitary, and one may check that one of the operators thus obtained coincides 
with the $U$ obtained from the polar decomposition of our original transform $\Phi$.

We note that B\'erard \cite{Ber2} assumes from the beginning of his construction 
that his matrix is orthogonal by imposing a restriction equivalent to the one just 
stated for $\alpha$ and $\gamma$. 
The cases $\alpha=0$, $\gamma=1$ and $\alpha=\gamma=1$ (for which the 
respective matrices are not orthogonal) correspond respectively to the mappings 
$T_3$ and $T_4$ considered in \cite{BCDS} Section~2.

\section{Isospectral domains for general elliptic operators} \label{Sdrumce5}

We will now generalize our construction from Section~\ref{Sdrumce3} to allow for 
general non-self-adjoint elliptic operators on $L_2(\Omega_1)$ and $L_2(\Omega_2)$.
These operators still have compact resolvent, but are in general not self-adjoint. 
Hence the original formulation involving isospectrality is not strong enough. 
It turns out, however, that the machinery of the previous sections works in exactly 
the same fashion to give the desired similarity of the operators in the general case.
We start with a basic lemma describing how elliptic differential sectorial forms transform 
under isometries.

\begin{lemma} \label{ldrumce501}
Let $\Omega \subset \Ri^d$ be an open set, 
$C = (c_{ij})_{i,j \in \{1,\ldots,d \} } \colon \Omega \to M_{d \times d}(\Ci)$
a bounded measurable map and $\tau$ an isometry.
Define $a \colon H^1(\Omega) \times H^1(\Omega) \to \Ci$ by 
\begin{equation}
a(u,v) 
= \int_\Omega \sum_{i,j=1}^d c_{ij} \, (\partial_i u) \, \overline{\partial_j v}
\label{ldrumce501;2}
\end{equation}
and $\widehat \Omega = \tau(\Omega)$.
Define the bounded measurable map
$C_\tau = \widehat C = (\hat c_{ij})_{i,j \in \{1,\ldots,d \} } \colon \widehat \Omega \to M_{d \times d}(\Ci)$
by
\begin{equation}
C_\tau(y) 
= \widehat C(y) 
= (D \tau) \, C(\tau^{-1}(y)) \, (D \tau)^{-1}
,  
\label{eldrumce501;1}
\end{equation}
where $D \tau$ denotes the derivative of $\tau$.
Define the form $\hat a \colon H^1(\widehat \Omega) \times H^1(\widehat \Omega) \to \Ci$
by 
\[
\hat a(u,v) 
= \int_{\widehat \Omega} \sum_{i,j=1}^d \hat c_{ij} \, (\partial_i u) \, \overline{\partial_j v}
.  \]
Then $\hat a(u,v) = a(u \circ \tau, v \circ \tau)$ for all 
$u,v \in H^1(\widehat \Omega)$.
\end{lemma}
\begin{proof}
Denote by $\langle \cdot , \cdot \rangle$ the inner product on $\Ci^d$.
Then 
\begin{eqnarray*}
a(u \circ \tau, v \circ \tau)
& = & \int_\Omega \langle C^t \nabla(u \circ \tau), \nabla (v \circ \tau) \rangle  \\
& = & \int_\Omega \langle C^t \, (D \tau)^t \, ((\nabla u) \circ \tau), 
      (D \tau)^t \, ((\nabla v) \circ \tau) \rangle  \\
& = & \int_\Omega \langle (D \tau) \, C^t \, (D \tau)^t \, ((\nabla u) \circ \tau), 
      ((\nabla v) \circ \tau) \rangle  \\
& = & \int_{\widehat \Omega} 
    \langle (D \tau) \, (C^t \circ \tau^{-1}) \, (D \tau)^t \, \nabla u, \nabla v \rangle  \\
& = & \hat a(u,v)
\end{eqnarray*}
as required.
\end{proof}

Let $\Omega \subset \Ri^d$ be an open set and 
$C = (c_{ij})_{i,j \in \{1,\ldots,d \} } \colon \Omega \to M_{d \times d}(\Ci)$
a bounded measurable map.
Define $a \colon H^1(\Omega) \times H^1(\Omega) \to \Ci$ by 
\[
a(u,v) 
= \int_\Omega \sum_{i,j=1}^d c_{ij} \, (\partial_i u) \, \overline{\partial_j v}
.  \]
Suppose that there exists a $\mu > 0$ such that 
\begin{equation}
\RRe \sum_{i,j=1}^d c_{ij} \, \xi_i \, \overline{\xi_j} 
\geq \mu \, |\xi|^2 \qquad \mbox{ for all } \xi \in \Ci^d
\label{eSdrumce5;1}
\end{equation}
almost everywhere on $\Omega$.
Then the form $a$ is elliptic.
Let $A$ be the operator associated with the form~$a$ on $L_2(\Omega)$.
Note that $A$ is self-adjoint if $c_{ij} = \overline{c_{ji}}$ a.e.\ for all 
$i,j \in \{ 1,\ldots,d \} $.
We emphasize that we do not assume this.
If $\Omega$ is bounded and Lipschitz, then by a result of Auscher--Tchamitchian
\cite{AT5} the form $a$ has the square root property on $L_2(\Omega)$.
Also note that if $C$ is the identity matrix, then $A$ is the Neumann Laplacian.
If $\tau$ is an isometry and $\widehat C$ is as in Lemma~\ref{ldrumce501}, 
then also $\widehat C$ is the identity matrix. 
So the Neumann Laplacian is transformed into the Neumann Laplacian, and 
the proof of (\ref{edrumce3b;1}) is a special case of the previous lemma.
This is one of the remarkable properties of the Laplacian.
However, if we consider an elliptic operator, then we have to take into 
account the conjugation with the derivative of the isometry.

Next, let $C$ be a bounded measurable elliptic matrix valued function on our reference triangle~$T$.
Thus 
\[
C = (c_{ij})_{i,j \in \{1,2 \} } \colon T \to M_{2 \times 2}(\Ci)
\]
is a bounded measurable map satisfying the ellipticity condition
(\ref{eSdrumce5;1}).
Let $\Omega_1$ and $\Omega_2$ be the two propellers as before. 
Define the form $a_1 \colon H^1(\Omega_1) \times H^1(\Omega_1) \to \Ci$ by 
\begin{equation}
a_1(u,v) = \sum_{k=1}^7 \int_{T_k} (C_{\tau_k})_{ij} \, (\partial_i u) \, \overline{\partial_j v}
,  \label{eSdrumce5;6}
\end{equation}
where for all $k \in \{ 1,\ldots,7 \} $ the isometry $\tau_k$ is as in Section~\ref{Sdrumce3}
and $C_{\tau_k}$ is defined in (\ref{eldrumce501;1}). 
We define the form $a_2 \colon H^1(\Omega_2) \times H^1(\Omega_2) \to \Ci$ 
analogously. 
Let $n \in \{ 1,2 \} $.
Then $a_n$ is elliptic.
Let $A_n$ be the operator associated with $a_n$ on $L_2(\Omega_n)$ and
let $S^n$ be the semigroup generated by $- A_n$ on $L_2(\Omega_n)$.
Next define the form $\tilde a \colon H^1(T) \times H^1(T) \to \Ci$ by 
\begin{equation}
\tilde a(u,v) = \sum_{i,j=1}^2 \int_T c_{ij} \, (\partial_i u) \, \overline{\partial_j v}
.  \label{eSdrumce5;5}
\end{equation}
Moreover, define the form $\tilde a_n \colon V_n \times V_n \to \Ci$ by 
\begin{equation}
\tilde a_n(u,v)
= \sum_{k=1}^7 \tilde a(u_k,v_k)
.  \label{eSdrumce5;7}
\end{equation}
Let $\widetilde A_n$ be the operator associated with $\tilde a_n$ on $L_2(T)^7$
and let $\widetilde S^n$ be the semigroup generated by $- \widetilde A_n$ on $L_2(T)^7$.

Using Lemma~\ref{ldrumce501} it follows as in Section~\ref{Sdrumce3} that 
\[
\Phi_n \, S^n_t \, (\Phi_n)^{-1} 
= \widetilde S^n_t
\]
for all $t > 0$, where $\Phi_n$ is the {\em same} transform as in 
Section~\ref{Sdrumce3}.
Arguing as in Section~\ref{Sdrumce3} one has 
\[
\Phi \, \widetilde S^1_t = \widetilde S^2_t \, \Phi
\]
for all $t > 0$, where surprisingly $\Phi$ is, again, the same transform as in 
Section~\ref{Sdrumce3}.

Therefore we have proved the following theorem.

\begin{thm} \label{tdrumce502}
Let $U = \Phi_2^{-1} \, \Phi \, \Phi_1$.
Then 
\[
S^1_t = U^{-1} \, S^2_t \, U
\]
for all $t > 0$.
In particular, the operators $A_1$ on $L_2(\Omega_1)$ and 
$A_2$ on $L_2(\Omega_2)$ are similar even though $\Omega_1$ and $\Omega_2$
are not congruent.
\end{thm}

\section{Isospectral elliptic operators with Dirichlet boundary conditions} \label{Sdrumce6}

In this section we wish to extend Theorem~\ref{tdrumce502} to the case
of Dirichlet boundary conditions.
All the arguments are the same as before, but now we have to impose more 
boundary conditions on the Sobolev spaces.
Let $\Omega \subset \Ri^d$ be open and 
$C = (c_{ij}) \colon \Omega \to M_{d \times d}(\Ci)$
a bounded measurable map.
Assume that $C$ satisfies the ellipticity condition (\ref{eSdrumce5;1}).
Let $a$ be as in (\ref{ldrumce501;2}) and let $a^D = a|_{H^1_0(\Omega) \times H^1_0(\Omega)}$,
where $H^1_0(\Omega)$ is the closure of $C_c^\infty(\Omega)$ in 
$H^1(\Omega)$.
Then the operator associated with $a^D$ in $L_2(\Omega)$ is the 
corresponding elliptic differential operator with Dirichlet boundary conditions.
For domains with Lipschitz boundary there is a useful characterization
of the Sobolev space $H^1_0(\Omega)$.

\begin{lemma} \label{ldrumce601}
Suppose $\Omega \subset \Ri^d$ is open and $\Omega$ has a Lipschitz boundary.
Then 
\[
H^1_0(\Omega)
= \{ u \in H^1(\Omega) : \Tr u = 0 \ \sigma \mbox{-a.e.} \}
.  \]
\end{lemma}
\begin{proof}
See \cite{Alt} Lemma~A.6.10.
\end{proof}

Now we return to the two propellers.
Let 
\[
C = (c_{ij})_{i,j \in \{1,2 \} } \colon T \to M_{2 \times 2}(\Ci)
\]
be a bounded measurable map satisfying the ellipticity condition
(\ref{eSdrumce5;1}).
Let $n \in \{ 1,2 \} $.
Let $a_n \colon H^1(\Omega_n) \times H^1(\Omega_n) \to \Ci$ be as in (\ref{eSdrumce5;6})
and set 
\[
a_n^D = a_n|_{H^1_0(\Omega_n) \times H^1_0(\Omega_n)}
.  \]
Let $A_n^D$ be the operator associated with $a_n^D$
on $L_2(\Omega_n)$ and let $S^{D,n}$ be the semigroup generated 
by $-A_n^D$.
Let $\Phi_n \colon L_2(\Omega_n) \to L_2(T)^7$ be as in Section~\ref{Sdrumce3}.
Define $V_n^D = \Phi_n(H^1_0(\Omega_n))$.
Let $\tilde a \colon H^1(T) \times H^1(T) \to \Ci$ be as in (\ref{eSdrumce5;5}).
Define $\tilde a_n^D \colon V_n^D \times V_n^D \to \Ci$ by 
\[
\tilde a_n^D(u,v)
= \sum_{k=1}^7 \tilde a(u_k,v_k)
.  \]
So $\tilde a_n^D = \tilde a_n|_{V_n^D \times V_n^D}$,
where $\tilde a_n$ is as in (\ref{eSdrumce5;7}).
Let $\widetilde A_n^D$ be the operator associated with $\tilde a_n^D$
on $L_2(T)^7$ and let $\widetilde S^{D,n}$ be the semigroup
generated by $- \widetilde A_n^D$ on $L_2(T)^7$.
Then as before one has 
\[
\Phi_n \, S^{D,n}_t \, \Phi_n^{-1} 
= \widetilde S^{D,n}_t
\]
for all $t > 0$.

Next we determine $V_n^D$.
Since Lemma~\ref{ldrumce601} imposes boundary conditions on the 
9 parts of the boundary of $\Omega_n$, one has 
\begin{eqnarray*}
V_1^D = \{ (u_1,\ldots,u_7) \in H^1(T)^7 
& : & u_1 = u_2 \mbox{ and } u_4 = u_7 \mbox{ and } u_3 = u_5 = u_6 = 0 \mbox{ on } \Gamma_1  \\
& & u_1 = u_3 \mbox{ and } u_2 = u_5 \mbox{ and } u_4 = u_6 = u_7 = 0  \mbox{ on } \Gamma_2  \\
& & u_1 = u_4 \mbox{ and } u_3 = u_6 \mbox{ and } u_2 = u_5 = u_7 = 0  \mbox{ on } \Gamma_3 \}
\end{eqnarray*}
and 
\begin{eqnarray*}
V_2^D = \{ (u_1,\ldots,u_7) \in H^1(T)^7 
& : & u_1 = u_2 \mbox{ and } u_3 = u_6 \mbox{ and } u_4 = u_5 = u_7 = 0  \mbox{ on } \Gamma_1  \\
& & u_1 = u_3 \mbox{ and } u_4 = u_7 \mbox{ and } u_2 = u_5 = u_6 = 0  \mbox{ on } \Gamma_2  \\
& & u_1 = u_4 \mbox{ and } u_2 = u_5 \mbox{ and } u_3 = u_6 = u_7 = 0  \mbox{ on } \Gamma_3 \}
.
\end{eqnarray*}
Now define $B^D = (b^D_{kl})_{k,l \in \{ 1,\ldots,7 \} } \colon \Ri^7 \to \Ri^7$ by 
\[
B^D = \left( \begin{array}{ccccccc}
 0 & 1 & 1 & 1 & 0 & 0 & 0  \\
 1 & 0 & -1 & 0 & 0 & 0 & 1  \\
 1 & 0 & 0 & -1 & 1 & 0 & 0  \\
 1 & -1 & 0 & 0 & 0 & 1 & 0  \\
 0 & 0 & 0 & 1 & 0 & -1 & -1  \\
 0 & 1 & 0 & 0 & -1 & 0 & -1  \\
 0 & 0 & 1 & 0 & -1 & -1 & 0  
           \end{array} \right)
\]
and define $\Phi^D \colon L_2(T)^7 \to L_2(T)^7$ by 
\[
(\Phi^D u)_k = \sum_{l=1}^7 b^D_{kl} \, u_l
.  \]
A surprising but simple calculation shows that 
\[
\Phi^D(V^D_1) \subset V^D_2 
\quad \mbox{and} \quad (\Phi^D)^*(V^D_2) \subset V^D_1
.  \]
Literally the same argument as in (\ref{eSdrumce4;1}) gives 
\[
\tilde a^D_2(\Phi^D u, v)
= \tilde a^D_1(u, (\Phi^D)^* v)
\]
for all $u \in V^D_1$ and $v \in V^D_2$.
Therefore Proposition~\ref{pdrumce202} implies that 
\[
\Phi^D \, \widetilde S^{D,1}_t = \widetilde S^{D,2}_t \, \Phi^D
\]
for all $t > 0$.
Hence we have extended everything for Dirichlet boundary conditions.

\begin{thm} \label{tdrumce602}
Let $U^D = \Phi_2^{-1} \, \Phi^D \, \Phi_1$.
Then 
\[
S^{D,1}_t = (U^D)^{-1} \, S^{D,2}_t \, U^D
\]
for all $t > 0$.
In particular, the operators $A^D_1$ on $L_2(\Omega_1)$ and 
$A^D_2$ on $L_2(\Omega_2)$ are isospectral even though $\Omega_1$ and $\Omega_2$
are not congruent.
\end{thm}

\section{Operators with Robin boundary conditions} \label{Sdrumce7}

If we again let $\Omega \subset \Ri^2$ be an arbitrary polygon, or more generally 
Lipschitz planar domain, with boundary $\Gamma$, then for all $\beta \in \Ri$ 
 we define a new form 
$a^\beta \colon H^1(\Omega) \times H^1(\Omega) \to \Ci$ by
\begin{equation}
a^\beta(u,v) = \int_\Omega \nabla u \cdot \overline{\nabla v} 
+ \beta \int_\Gamma u\,\overline{v}.
\label{eSdrumce7;1}
\end{equation}
It follows from the Trace Inequality that the form $a^\beta$ is continuous and 
$L_2(\Omega)$-elliptic for all $\beta \in \Ri$.
We denote by $-\Delta^\beta_\Omega$ the operator on $L_2(\Omega)$ associated 
with $a^\beta$ and call $\Delta^\beta_\Omega$ the {\bf Robin Laplacian} 
with boundary coefficient $\beta$, which has domain given by
\[
D(\Delta^\beta_\Omega) = \{u \in H^1(\Omega):\Delta u\in L_2(\Omega) 
\mbox{ and } \partial_\nu u+\beta u=0 \mbox{ on } \Gamma\}
. \]
In the boundary condition $\partial_\nu u+\beta u=0$, the normal derivative 
$\partial_\nu u$ is defined by {\rm (\ref{eSdrumce3;0})}, as in the case of the 
Neumann Laplacian, and by $u$ we mean the trace of $u$ on $\Gamma$.
As is true of its Dirichlet and Neumann counterparts, the Robin 
Laplacian is self-adjoint and has compact resolvent, and its negative is bounded 
from below.
We denote  by $S^\beta$ the semigroup generated by $\Delta^\beta_\Omega$.
When $\beta=0$ we recover the Neumann Laplacian, and for 
$\beta \in (0,\infty)$, the Robin Laplacian `interpolates' between the 
Dirichlet and Neumann Laplacians in a strong sense
\cite{AW}.

The boundary condition $\partial_\nu u+\beta u=0$ corresponds to an 
`elastically supported membrane'.
So if we interpret the Dirichlet boundary condition as representing a drum 
with a taut membrane and the Neumann condition naively as representing 
a gong, then the Robin condition describes a drum whose membrane is 
not properly attached to the body of the drum, but rather allowed to 
move a little as the membrane vibrates.

Our goal is to show that no operator formed as a sum of superimposing isometries 
between component triangles of $\Omega_1$ and $\Omega_2$ (as in the 
Neumann and Dirichlet cases) can intertwine the Robin Laplacians 
$\Delta^\beta_{\Omega_1}$ and $\Delta^\beta_{\Omega_2}$ for any 
$\beta \neq 0$.
We make this statement precise by recalling some notation from 
Section~\ref{Sdrumce3}.
If we first consider $\Omega_1$, we recall that $\Phi_1 \colon L_2 (\Omega_1) \to 
L_2(T)^7$ is the unitary operator associated with the family of isometries 
$\tau_k \colon T \to T_k$, such that
\[
\Phi_1(w) = (w|_{T_1} \circ \tau_1,\ldots, w|_{T_7} \circ \tau_7)
\]
for all $w \in L_2(\Omega_1)$, and moreover $\Phi_1 (H^1(\Omega_1)) = V_1$.
Note that since the Robin Laplacian has the same form domain as the Neumann 
Laplacian, $\Phi_1$ is still the correct operator to use in this case.
However, we now wish to consider the image of $\partial\Omega_1$, the 
boundary of $\Omega_1$, under the isometries $\tau_k$.
We write
\[
\Gamma^1_k := \tau_k^{-1}(\partial \Omega_1 \cap \overline T_k)
\]
for all $k \in \{ 1,\ldots,7 \} $.
Then $\Gamma^1_k \subset \Gamma_1 \cup \Gamma_2 \cup \Gamma_3$; 
for example, $\Gamma^1_1 = \emptyset$ and $\Gamma^1_5 = \Gamma_1 
\cup \Gamma_3$.
(Cf. Figures~\ref{fdrumce1} and \ref{fdrumce2}.)
For fixed $\beta \in \Ri$ we now define a form $\tilde a^\beta_1 \colon  V_1 
\times V_1 \to \Ci$ by
\[
\tilde a^\beta_1 (\Phi_1 u, \Phi_1 v) = a^\beta_1(u,v)
\]
for $u,v \in H^1(\Omega_1)$.
Then
\[
\tilde a^\beta_1 (u,v) = \sum_{k=1}^7 \int_T \nabla u_k \cdot 
\overline{\nabla v_k} + \beta \int_{\Gamma^1_k} u_k \,\overline{v_k}
\]
for all $u = (u_1,\ldots, u_7), v = (v_1,\ldots, v_7) \in V_1$.
Note that  $\tilde a^0_1$ coincides with the form $\tilde a_1$ introduced in 
(\ref{epdrumce3b01}).
We do the same for $\Omega_2$, so that the unitary operator 
$\Phi_2 \colon  L_2(\Omega_2) \to L_2(T)^7$ intertwines the forms $a^\beta_2$ 
and $\tilde a^\beta_2 \colon  V_2 \times V_2 \to \Ci$ given by
\[
\tilde a^\beta_2 (u,v) = \sum_{k=1}^7 \int_T \nabla u_k \cdot 
\overline{\nabla v_k} + \beta \int_{\Gamma^2_k} u_k \,\overline{v_k}.
\]
For an arbitrary invertible matrix $P \colon  \Ri^7 \to \Ri^7$ given by 
$P=(p_{kl})$ we construct an associated operator $\Phi \colon  L_2(T)^7 \to 
L_2(T)^7$ by setting
\begin{equation}
(\Phi u)_k = \sum_{l=1}^7 p_{kl} \, u_l \qquad \mbox{and} \quad (\Phi^* u)_l 
= \sum_{k=1}^7 p_{kl} \, u_k,
\label{eSdrumce7;2}
\end{equation}
where $(u_1,\ldots,u_7) \in L_2(T)^7$.
We will prove that, for any $\beta \neq 0$, there is no matrix $P$ such that 
the associated operator $\Phi \colon  L_2(T)^7 \to L_2(T)^7$ satisfies $\Phi(V_1) 
\subset V_2$, $\Phi^*(V_2) \subset V_1$ and
\begin{equation}
\tilde a^\beta_2(\Phi u,v) = \tilde a^\beta_1(u, \Phi^* v)
\label{eSdrumce7;3}
\end{equation}
for all $u\in V_1$ and $v\in V_2$.
Since this is equivalent to the non-existence of an operator 
$U = \Phi_2^{-1} \, \Phi \, \Phi_1 \colon  L_2(\Omega_1) \to L_2(\Omega_2)$ intertwining 
$a^\beta_1$ and $a^\beta_2$, the impossibility of (\ref{eSdrumce7;3}) 
then implies via Proposition~\ref{pdrumce202} that the Robin Laplacians 
cannot be intertwined by an operator expressible as a sum of isometries 
between the triangles.
To that end, we first show that the question as to whether (\ref{eSdrumce7;3}) holds
is independent of the coefficient $\beta \neq 0$.

\begin{prop} \label{pdrumce701}
Let $\Phi \in \cl (L_2(T)^7, L_2(T)^7)$ be defined by 
{\rm (\ref{eSdrumce7;2})} and satisfy $\Phi(V_1) \subset V_2$ 
and $\Phi^*(V_2) \subset V_1$. 
If {\rm (\ref{eSdrumce7;3})}
holds for some $\beta \in \Ri \setminus \{0\}$ , then the same is 
true for all $\beta \in \mathbb{R}$.
\end{prop}
\begin{proof}
This follows easily from the definition (\ref{eSdrumce7;2}) of the operator $\Phi$.
Just as in the Neumann case (cf.~(\ref{eSdrumce4;1})), we have
\begin{equation}
\tilde a^\beta_2(\Phi u, v)=  \sum_{k=1}^7 \sum_{l=1}^7 p_{kl} 
\Big( \int_T \nabla u_l \cdot \overline{\nabla v_k}
   +\beta\int_{\Gamma^2_k}u_l \, \overline{v_k}\Big),
\label{pdrumce701-2}
\end{equation}
while
\begin{equation}
\tilde a^\beta_1(u, \Phi^* v) = \sum_{k=1}^7 \sum_{l=1}^7 p_{kl} 
\Big( \int_T \nabla u_l \cdot \overline{\nabla v_k}
   +\beta\int_{\Gamma^1_l}u_l \, \overline{v_k}\Big).
\label{pdrumce701-3}
\end{equation}
By assumption, the two are equal, and so
\begin{equation}
\beta \sum_{k=1}^7 \sum_{l=1}^7 p_{kl} \Big(\int_{\Gamma^2_k}u_l \, \overline{v_k}
  - \int_{\Gamma^1_l} u_l \, \overline{v_k} \Big)=0.
\label{pdrumce701-4}
\end{equation}
Since $\beta \neq 0$, if we take any other $\beta_0 \in \Ri$ and multiply 
(\ref{pdrumce701-4}) by $\beta_0/\beta$, we see from (\ref{pdrumce701-2}) 
and (\ref{pdrumce701-3}) applied to $\beta_0$ that (\ref{eSdrumce7;3}) 
must hold for $\beta_0$.
\end{proof}

Our next result, which applies to any bounded Lipschitz domains $\omega_1$ and 
$\omega_2$ in $\Ri^d$, 
states that if a unitary operator $U$ intertwines two 
Robin Laplacians for two separate values of $\beta\in\Ri$, then the {\em same} 
operator intertwines the Robin Laplacians for {\em all} values of $\beta\in\Ri$, 
including the Neumann Laplacians, as well as the Dirichlet Laplacians, and 
also acts isometrically on the traces of functions in $H^1(\omega_1)$.

\begin{prop} \label{pdrumce702}
Let $\omega_1$ and $\omega_2$ be bounded Lipschitz domains in $\Ri^d$.
For all $\beta \in \Ri$ denote by $a^\beta_1$ and $a^\beta_2$ the 
forms given by {\rm (\ref{eSdrumce7;1})} on $\omega_1$ and $\omega_2$.
Suppose $U \in \cl(L_2(\omega_1), L_2(\omega_2))$ is unitary, with 
$U (H^1(\omega_1)) = H^1(\omega_2)$. The following statements 
are equivalent.
\begin{tabeleq}
\item \label{pdrumce702-1} 
There exist $\beta_1,\beta_2 \in \Ri$ with 
$\beta_1 \neq \beta_2$ such that
\[
a^{\beta_n}_1 (u,v) = a^{\beta_n}_2(U u, U v)
\]
for all $u,v\in H^1(\omega_1)$ and $n \in \{ 1,2 \} $.
\item \label{pdrumce702-2} 
The operator $U$ intertwines the Neumann Laplacians on 
$\omega_1$ and $\omega_2$. Moreover, if $u,v \in H^1(\omega_1)$, then 
\begin{equation}
	\int_{\partial\omega_1} u\, \overline v = \int_{\partial\omega_2} (U u) \, \overline{(U v)},
\label{pdrumce702-3}
\end{equation}
where by $u$ we mean the trace of $u$, etc.
\end{tabeleq}
Moreover, if these equivalent conditions are satisfied, then 
$U(H^1_0(\omega_1))=H^1_0(\omega_2)$ and $U$ intertwines the 
Dirichlet Laplacians on $\omega_1$ and $\omega_2$.
\end{prop}
\begin{proof}
`\ref{pdrumce702-1}$\Rightarrow$\ref{pdrumce702-2}'.
By writing out the form condition in \ref{pdrumce702-1} for $\beta_1$ and 
$\beta_2$ and taking the difference of the two expressions, we obtain 
directly that
\begin{equation}
\int_{\omega_1} \nabla u \cdot \overline{\nabla v}
= \int_{\omega_2} \nabla (U u)\cdot \overline{\nabla (U v)}
\quad\mbox{and}\quad
\int_{\partial\omega_1}u\, \overline v 
= \int_{\partial\omega_2} (U u) \, \overline{(U v)}
\label{pdrumce702-4}
\end{equation}
for all $u,v \in H^1(\omega_1)$.
It follows immediately from Corollary~\ref{cdrumce204} that $U$ 
intertwines the Neumann Laplacians on $\omega_1$ and $\omega_2$.

`\ref{pdrumce702-2}$\Rightarrow$\ref{pdrumce702-1}'.
Fix $\beta \in \Ri$ and $u,v \in H^1(\omega_1)$.
Since $U$ intertwines the Neumann Laplacians, by Corollary~\ref{cdrumce204} 
it intertwines the associated forms.
Therefore
\begin{equation}
\int_{\omega_1} \nabla u \cdot \overline{\nabla v}
= \int_{\omega_2} \nabla (U u) \cdot \overline{\nabla (U v)}.
\label{pdrumce702-5}
\end{equation}
Moreover, since by assumption (\ref{pdrumce702-3}) holds, it follows directly 
from the definition (\ref{eSdrumce7;1}) of $a^{\beta_n}$ that 
\ref{pdrumce702-1} holds for all $\beta_1,\beta_2 \in \Ri$.

Finally, to prove the last assertion, suppose that $w \in H^1_0(\omega_1)$.
Then (\ref{pdrumce702-4}) applied to $u=v=w$ implies that
\[
\int_{\partial\omega_2} |U w|^2 = \int_{\partial\omega_1} |w|^2 = 0.
\]
By Lemma~\ref{ldrumce601}, it follows that $U w \in H^1_0(\Omega_2)$.
Thus $U(H^1_0(\omega_1)) \subset H^1_0(\omega_2)$.
Since $U^{-1} = U^*$ has exactly the same properties as $U$, an identical  
argument shows that $U^{-1}(H^1_0(\omega_2)) \subset H^1_0(\omega_1)$ 
and therefore $U(H^1_0(\omega_1)) = H^1_0(\omega_2)$.
Moreover, it is clear that $U|_{H^1_0(\omega_1)}$ is still a continuous 
linear bijection, and since (\ref{pdrumce702-5}) holds for all 
$u,v \in H^1_0(\omega_1) \subset H^1(\omega_1)$, by 
Corollary~\ref{cdrumce204} this means $U$ intertwines the Dirichlet 
Laplacians.
\end{proof}

We now show that no one operator $\Phi$ of the form (\ref{eSdrumce7;2}) 
can simultaneously intertwine the Dirichlet and Neumann Laplacians, 
which is a noteworthy  
observation in its own right.
It is also worth noting that it can be proved by observing that the families of 
matrices $\widehat B = \alpha \one- \gamma B$ and $\widehat B^D := \alpha 
\one- \gamma B^D$, for nontrivial combinations of $\alpha$ and $\gamma$,
are the only ones giving rise to operators intertwining the Neumann and Dirichlet 
Laplacians, respectively, and they have no matrix in common.
However, we give a different proof based on reflections.
The principle is that, if one reflects a triangle $T$ along one of its sides, and 
wishes to reflect functions in $H^1(T)$ across to the larger domain, one does 
so by taking even reflections along the common line.
But to preserve $H^1_0(T)$ the reflection should be odd.

\begin{prop} \label{pdrumce703}
No invertible operator $\Phi \colon  L_2(T)^7 \to L_2(T)^7$ of the form 
{\rm (\ref{eSdrumce7;2})} simultaneously satisfies the Neumann condition
\[
\Phi(V_1) \subset V_2 \quad \mbox{and} \quad \Phi^*(V_2) \subset V_1
\]
and the Dirichlet condition
\[
\Phi(V_1^D) \subset V_2^D \quad \mbox{and} \quad \Phi^*(V_2^D) \subset V_1^D.
\]
\end{prop}
\begin{proof}
Assume $\Phi$ is associated with $P = (p_{kl}) \colon \Ri^7 \to \Ri^7$.

Consider $m := p_{12}$.
Let $w \in C_c^\infty(T \cup \Gamma_3)$ be such that $w$ does not vanish 
identically on $\Gamma_3$. 
If we define $u = (0,w,0,\ldots,0)$, then it is easily checked that $u \in V_1$. 
Moreover, we have $(\Phi u)_1 = m \, w$ and $(\Phi u)_4 = p_{42} \, w$, using 
the definition {\rm (\ref{eSdrumce7;2})} of $\Phi$.
But since $\Phi u \in V_2$, we must have $(\Phi u)_1 = (\Phi u)_4$ on 
$\Gamma_3$ in the sense of traces.
Since $w \not\equiv 0$ on $\Gamma_3$, this means $p_{42}=m$.
Alternatively, choose $v = (w,0,0,w,0,0,0)$. 
Then $v \in V_2^D$.
Moreover, $\Phi^* v = (0, 2m \, w, 0,0,0,0,0)$.
But $\Phi^* v \in V_1^D$ by assumption.
So $2 m \, w$ vanishes on $\Gamma_3$.
This implies that $p_{12} = m = 0$.

Arguing similarly, it follows that $p_{kl} = 0$ for all 
$(k,l) \in \{ 1,\ldots,7 \} ^2 \setminus S$, where
$S = \{ (1,1), (2,4), (3,2), (4,3), (5,6), (6,7), (7,5) \} $.
Since $P$ is invertible, one has $p_{kl} \neq 0$ for all 
$(k,l) \in S$.
Then 
\[
\Phi u 
= (p_{11} \, u_1, 
   p_{24} \, u_4, 
   p_{32} \, u_2, 
   p_{43} \, u_3, 
   p_{56} \, u_6, 
   p_{67} \, u_7, 
   p_{75} \, u_5)
.  \]
If $w$ is as above, but one chooses this time
$u = (w,0,0,w,0,0,0)$, then $u \in V_1$.
So $\Phi u \in V_2$ by assumption.
Hence $(\Phi u)_1 = (\Phi u)_4$ on $\Gamma_3$, which implies that 
$p_{11} \, w = 0$ on $\Gamma_3$. 
This is a contradiction.
\end{proof}

Note that the same proof also works if $P$ has complex coefficients.
Our main result, that the Robin Laplacians on $\Omega_1$ and $\Omega_2$ 
are not intertwined by any operator acting as a linear combination of isometries 
between triangles, now follows easily.

\begin{thm} \label{tdrumce701}
Suppose $\beta \neq 0$. 
Then there does not exist an invertible operator $\Psi \colon  L_2(\Omega_1) 
\to L_2(\Omega_2)$ of the form $\Psi = \Phi_2^{-1} \, \Phi \, \Phi_1$, where 
$\Phi \colon  L_2(T)^7 \to L_2(T)^7$ is of the form {\rm (\ref{eSdrumce7;2})}, 
which intertwines $\Delta^\beta_{\Omega_1}$ and $\Delta^\beta_{\Omega_2}$.
\end{thm}
\begin{proof}
Suppose that there does exist such a $\Psi$, and therefore a $\Phi$ associated with 
some invertible operator $P \colon \Ri^7 \to \Ri^7$. 
By using the polar decomposition of $P$ (see Section~\ref{Sdrumce4}), we may 
assume without loss of generality that $P$ and therefore also $\Phi$ and $\Psi$ are 
unitary.
By Proposition~\ref{pdrumce701}, the map $\Phi$ satisfies (\ref{eSdrumce7;3}) 
for all $\beta \in \Ri$ and therefore $\Psi$ intertwines both the Neumann 
and Dirichlet Laplacians on $\Omega_1$ and $\Omega_2$  by Proposition~\ref{pdrumce702}.
But this contradicts Proposition~\ref{pdrumce703}.
\end{proof}

It is clear that the same method of proof works not only for more general elliptic operators, 
but also for all known planar counterexamples, and indeed, should still be true for 
all pairs of (Dirichlet or Neumann) isospectral domains for which Sunada's principle 
applies.
In particular, there are no known pairs of noncongruent domains for which the Robin 
Laplacians are isospectral (for any $\beta \neq 0$), and there is no reason 
to suppose that any known Dirichlet or Neumann counterexamples 
have this property.

\subsection*{Acknowledgements}
The authors wish to thank Moritz Gerlach for providing the pictures.
The second named author is most grateful for the hospitality extended
to him during a fruitful stay at the University of Ulm.
He wishes to thank the University of Ulm for financial support.
Part of this work is supported by the Marsden Fund Council from Government funding, 
administered by the Royal Society of New Zealand.
The third named author is supported by a fellowship of the 
Alexander von Humboldt Foundation, Germany.


\begin{thebibliography}{ABHN}

\bibitem[Alt]{Alt}
{\sc Alt, H.~W.}, {\em Lineare Funktionalanalysis}.
\newblock Springer-Verlag, Berlin etc., 1985.

\bibitem[Are]{Are3}
{\sc Arendt, W.}, Does diffusion determine the body?
\newblock {\em J. Reine Angew.\ Math.} {\bf 550} (2002),  97--123.

\bibitem[ABHN]{ABHN}
{\sc Arendt, W., Batty, C., Hieber, M. {\rm and} Neubrander, F.}, {\em
  Vector-valued Laplace transforms and Cauchy problems}, vol.\ 96 of Monographs
  in Mathematics.
\newblock Birkh{\"a}user, Basel, 2001.

\bibitem[ABE]{ABE}
{\sc Arendt, W., Biegert, M. {\rm and} Elst, A. F.~M. ter}, Diffusion
  determines the manifold.
\newblock {\em J. Reine Angew.\ Math.} {\bf 667} (2012),  1--25.

\bibitem[AE]{AE4}
{\sc Arendt, W. {\rm and} Elst, A. F.~M. ter}, Diffusion determines the compact
  manifold.
\newblock {\em Banach Center Publ.} (2012).
\newblock In press, arXiv: 1104.1012.

\bibitem[AW]{AW}
{\sc Arendt, W. {\rm and} Warma, M.}, Dirichlet and Neumann boundary
  conditions: What is in between?
\newblock {\em J. Evol.\ Equ.} {\bf 3} (2003),  119--135.

\bibitem[AT]{AT5}
{\sc Auscher, P. {\rm and} Tchamitchian, P.}, Square roots of elliptic second
  order divergence operators on strongly Lipschitz domains: $L^2$ theory.
\newblock {\em J. Anal. Math.} {\bf 90} (2003),  1--12.

\bibitem[B\'er1]{Ber1}
{\sc B{\'e}rard, P.}, Transplantation et isospectralit\'e. I.
\newblock {\em Math. Ann.} {\bf 292} (1992),  547--559.

\bibitem[B\'er2]{Ber2}
\leavevmode\vrule height 2pt depth -1.6pt width 23pt, Domaines plans
  isospectraux \`a la Gordon-Webb-Wolpert: une preuve \'el\'emen\-taire.
\newblock {\em Afrika Mat. {\rm (3)}} {\bf 1} (1993),  135--146.

\bibitem[BCDS]{BCDS}
{\sc Buser, P., Conway, J., Doyle, P. {\rm and} Semmler, K.-D.}, Some planar
  isospectral domains.
\newblock {\em Internat. Math. Res. Notices} {\bf 9} (1994),  391--400.

\bibitem[Cha]{Cha}
{\sc Chapman, S.~J.}, Drums that sound the same.
\newblock {\em Amer. Math. Monthly} {\bf 102} (1995),  124--138.

\bibitem[GWW]{GWW}
{\sc Gordon, C., Webb, D.~L. {\rm and} Wolpert, S.}, One cannot hear the shape
  of a drum.
\newblock {\em Bull.\ Amer.\ Math.\ Soc.} {\bf 27} (1992),  134--138.

\bibitem[Kac]{Kac}
{\sc Kac, M.}, Can one hear the shape of a drum?
\newblock {\em Amer.\ Math.\ Monthly} {\bf 73} (1966),  1--23.

\bibitem[McI]{McI1}
{\sc McIntosh, A.}, On the comparability of $A^{1 / 2}$ and $A^{*1 / 2}$.
\newblock {\em Proc.\ Amer.\ Math.\ Soc.} {\bf 32} (1972),  430--434.

\bibitem[Pro]{Pro}
{\sc Protter, M.~H.}, Can one hear the shape of a drum? revisited.
\newblock {\em SIAM Rev.} {\bf 29} (1987),  185--197.

\bibitem[Sun]{Sun1}
{\sc Sunada, T.}, Riemannian coverings and isospectral manifolds.
\newblock {\em Ann. of Math.} {\bf 121} (1985),  169--186.

\bibitem[Ura]{Ura}
{\sc Urakawa, H.}, Bounded domains which are isospectral but not congruent.
\newblock {\em Ann.\ Sci.\ \'Ecole Norm.\ Sup.\ (4)} {\bf 15} (1982),
  441--456.

\bibitem[Wat]{Wat}
{\sc Watanabe, K.}, Plane domains which are spectrally determined.
\newblock {\em Ann. Global Anal. Geom.} {\bf 18} (2000),  447--475.

\bibitem[Yos]{Yos}
{\sc Yosida, K.}, {\em Functional Analysis}.
\newblock Sixth edition, Grundlehren der mathematischen Wissenschaften 123.
  Springer-Verlag, New York etc., 1980.

\bibitem[Zel1]{Zel}
{\sc Zelditch, S.}, Spectral determination of analytic bi-axisymmetric plane
  domains.
\newblock {\em Geom.\ Funct.\ Anal.} {\bf 10} (2000),  628--677.

\bibitem[Zel2]{Zel2}
\leavevmode\vrule height 2pt depth -1.6pt width 23pt, Inverse spectral problem
  for analytic domains. II. $\Zi_2$-symmetric domains.
\newblock {\em Ann. of Math.} {\bf 170} (2009),  205--269.

\end{thebibliography}
\end{document}